\documentclass[11pt]{article}

\usepackage{geometry}   
\usepackage{graphicx}
\usepackage{amsmath}
\usepackage{amssymb}
\usepackage{authblk}
\usepackage{amsthm}
\usepackage{epstopdf}
\usepackage{enumitem}
\usepackage{lscape}
\usepackage{ stmaryrd }
\usepackage{array}
\usepackage{adjustbox}
\usepackage{multirow}
\usepackage{bm}
\usepackage{tikz}
\usetikzlibrary{matrix, arrows, decorations.pathreplacing, shapes}
\setlist[description]{leftmargin=\parindent,labelindent=\parindent}

\numberwithin{equation}{section}

\newtheorem{thm}{Theorem}[section]
\newtheorem{lem}[thm]{Lemma}
\newtheorem{prop}[thm]{Proposition}

\theoremstyle{remark}
\newtheorem{rmk}[thm]{Remark}

\newtheorem{eg}[thm]{Example}
\theoremstyle{definition}
\newtheorem{defn}[thm]{Definition}

\newcommand{\Pf}{\operatorname{Pf}}

\newcommand{\Proj}{\operatorname{Proj}}
\newcommand{\codim}{\operatorname{codim}}

\newcommand{\Cl}{\operatorname{Cl}}

\newcommand{\Newt}{\operatorname{Newt}}

\newcommand{\sB}{\mathcal{B}}
\newcommand{\sI}{\mathcal{I}}
\newcommand{\sO}{\mathcal{O}}

\newcommand{\CC}{\mathbb{C}}

\newcommand{\PP}{\mathbb{P}}
\newcommand{\QQ}{\mathbb{Q}}
\newcommand{\ZZ}{\mathbb{Z}}

\newcommand{\frakm}{\mathfrak{m}}

\title{Constructing $\QQ$-Fano 3-folds \\ \`a la Prokhorov \& Reid}
\author{Tom Ducat}
\affil{RIMS, Kyoto University, Kyoto 606-8502, Japan. \\ \tt{taducat@kurims.kyoto-u.ac.jp}}
\date{\today}     
                       
\begin{document}

\maketitle

\begin{abstract}
We generalise a construction by Prokhorov \& Reid of two families of $\QQ$-Fano 3-folds of index 2 to obtain five more families of $\QQ$-Fano 3-folds; four of index 2 and one of index 3. Two of the families constructed have the same Hilbert series and we study these cases in more detail.
\end{abstract}

\section*{Introduction}
\addcontentsline{toc}{section}{Introduction}

\paragraph{$\QQ$-Fano 3-folds.} In this paper a \emph{$\QQ$-Fano 3-fold} $X$ will be a normal projective 3-dimensional variety over $\CC$ with $-K_X$ ample and at worst $\QQ$-factorial terminal singularities. Unless otherwise stated, we will assume that $X$ has Picard rank $\rho_X=1$ (so that $X$ appears as the end product of a Minimal Model Program). 

The \emph{($\QQ$-Fano) index} of $X$ is given by:
\[ q_X := \max{ \left\{ q\in\ZZ_{\geq1} : -K_X = qA \text{ for some }A\in\Cl(X) \right\} } \]  
and any Weil divisor $A$ for which $-K_X=q_XA$ is called a \emph{primitive ample divisor} on $X$. We usually consider $X$ to be polarised by $A$, i.e.\ with an embedding into weighted projective space given by $\Proj$ of the graded ring $R(X,A)=\bigoplus_{k\geq0}H^0(X,\sO_X(kA))$. Some of the basic numerical invariants for $(X,A)$ are the \emph{codimension} of this embedding, the \emph{index} $q_X$, the \emph{degree} $A^3$ and the \emph{basket of singularities} $\sB_X$.

\paragraph{Prokhorov \& Reid's construction.} 
Starting from either $X=\PP^3$ or $X=X_2\subset\PP^4$, Prokhorov \& Reid \cite{prokreid} \S6.3.3 construct a $\QQ$-Fano 3-fold $Y$ by writing down a Sarkisov link which makes a divisorial extraction from a certain kind of singular curve $\Gamma\subset X$ followed by the Kawamata blowdown of a divisor $E\cong\PP(1,1,2)$ to a (polarised) $\tfrac13(2,2,1)$ cyclic quotient singularity $Q\in Y$. 

The two families of $\QQ$-Fano 3-folds $(Y,B)$ constructed in this way correspond to the cases A.1 and A.2 of Table \ref{Ytable}, which lists their numerical data as well as their ID in the Graded Ring Database \cite{grdb}. In particular they both have index 2.

\paragraph{The main result.}
We generalise Prokhorov \& Reid's construction to obtain families for all of the other cases of Table \ref{Ytable}.

\begin{thm} \label{mainThm}
For each of the cases A.1-4 of Table \ref{Ytable} there is a Sarkisov link starting from the corresponding $\QQ$-Fano 3-fold $X$ of Table \ref{Xtable} and ending with $Y$. This Sarkisov link starts with a divisorial extraction from an irreducible singular curve $\Gamma\subset X$ followed by the Kawamata blowdown to a cyclic quotient singularity $Q\in Y$ (the $\tfrac15(1,3,4)$ point in case A.4). 

The cases B.1-2 can also be constructed by the exactly the same method, however we must start with a (non-extremal) divisorial extraction from a \emph{reducible} curve $\Gamma\subset X$. Therefore the $Y$ that are constructed in this manner have Picard rank $\rho_Y>1$.

In case A.3 we construct two different families corresponding to divisorial extractions from two types of non-isomorphic curve singularity $P\in \Gamma\subset X$. For all of the other cases the family constructed is unique.
\end{thm}

These are all of the possible examples that can be obtained by this method in its current form. In \S3 we study the two cases of A.3 in more detail.

\begin{table}
\arraycolsep=1.2pt\def\arraystretch{1.6}
\caption{The $\QQ$-Fano 3-folds $(Y,B)$ constructed in this paper.}
\begin{center}
\begin{tabular}{c|cccccm{2.2cm}c}
Case & $Y$ & $\rho_Y$ & $\codim$ & $q_Y$ & $B^3$ & \centering  $\sB_Y$ & GRDB \\ \hline
A.1 & $Y\subset \PP(1^4,2^2,3)$ & 1 & 3 & 2 & $\tfrac73$ & \centering $\tfrac13(1,2,2)$ & \#40836 \\
A.2 & $Y\subset \PP(1^5,2^2,3)$ & 1 &  4 & 2 & $\tfrac{10}3$ & \centering  $\tfrac13(1,2,2)$ & \#40933 \\ 
A.3 & $Y\subset \PP(1^3,2^2,3,4,5)$ & 1 &  4 & 2 & $\tfrac75$ & \centering  $\tfrac15(1,2,4)$ & \#40663 \\
A.4 & $Y\subset \PP(1^2,2^2,3^2,4,5)$ & 1 &  4 & 3 & $\tfrac35$ & \centering  $2\times\tfrac12(1,1,1)$, $\tfrac15(1,3,4)$ & \#41200 \\ \hline
B.1 & $Y\subset \PP(1^4,2^2,3,4,5)$ & 4 &  5 & 2 & $\tfrac{12}5$ & \centering  $\tfrac15(1,2,4)$ & \#40837 \\ 
B.2 & $Y\subset \PP(1^3,2^2,3,4,5,6,7)$ & 5 &  6 & 2 & $\tfrac{10}7$ & \centering  $\tfrac17(1,2,6)$ & \#40664  
\end{tabular}
\end{center}
\label{Ytable}
\end{table}%

\paragraph{Acknowledgements.}

The author is a International Research Fellow of the Japanese Society for the Promotion of Science and this work was supported by Grant-in-Aid for JSPS Fellows, No.\ 15F15771. I also thank Miles Reid for all his help and advice.

\section{Prokhorov \& Reid's construction}

We now outline our generalisation of Prokhorov \& Reid's construction \cite{prokreid} \S6.3.3, following their paper very closely. The original construction was only for the weighted projective plane $E=\PP(1,1,2)$.

\paragraph{The Sarkisov link.} Starting from a $\QQ$-Fano 3-fold $X$ the general idea is to construct a new $\QQ$-Fano 3-fold $Y$ by writing down a simple Sarkisov link of type II:
\begin{center}\begin{tikzpicture}
	\node at (0,0) {$X$};
	\node at (1,1) {$X'$};
	\node at (2,0) {$Y$};
	
	\draw[->] (0.7,0.7) -- node[above left] {$\sigma$} (0.3,0.3);
	\draw[->] (1.3,0.7) -- node[above right] {$\pi$} (1.7,0.3);
\end{tikzpicture}\end{center}
where $\sigma \colon X' \to X$ is a Mori extraction from a curve $\Gamma\subset X$ and $\pi \colon X' \to Y$ is the Kawamata blowup of a terminal cyclic quotient singularity $Q\in Y$. Prokhorov \& Reid construct such a Sarkisov link by the following method:

\paragraph{Step 1.} Take the weighted projective plane $E=\PP(1,r,ra-1)$ for some $r\geq2$, $a\geq1$. We consider weights of this form for the following two reasons:
\begin{itemize}
\item $E$ has a $\tfrac1r(1,-1)$ singularity (i.e.\ a type $A_{r-1}$ Du Val singularity),
\item $E$ is isomorphic to the exceptional divisor in the Kawamata blowup of the terminal cyclic quotient singularity $\tfrac1{ra+r-1}(1,r,ra-1)$. (In fact $E$ will become the exceptional divisor for $\pi$.) 
\end{itemize}

\paragraph{Step 2.} Consider the embedding of $E$ given by $\phi = \phi_{|\sO_E(r)|}$, i.e.\
\[ \phi \colon E \hookrightarrow \PP(1,1,a,ra-1),  \quad \phi(e_0:e_1:e_2) = (e_1:e_0^r:e_0e_2:e_2^r) \]
where the image $\phi(E)\subset\PP(1,1,a,ra-1)_{u:x:y:z}$ is given by the equation $xz=y^r$. We consider $\QQ$-Fano 3-folds $(X,A)$ with an embedding $E\hookrightarrow X$ such that $A|_E=\sO_E(r)$. The point of considering such embeddings is that the $\tfrac1r(1,-1)$ singularity of $E$ is supported at a smooth point $P=P_{u}\in X$.

\begin{lem}
Let $(X,A)$ be a $\QQ$-Fano 3-fold which admits an embedding $E\subset X$ as a $\QQ$-Cartier divisor such that $E\in |eA|$ and $\sO_E(A|_E)=\sO_E(r)$. Then $X$ is of the form:
\[ \text{(i)} \quad \PP(1,1,a,ra-1) \quad \text{or} \quad \text{(ii)} \quad X_{ra} \subset \PP(1,1,a,ra-1,e). \]
Moreover all possible choices are explicitly listed in Table \ref{possibleCases}.
\label{standardForm}
\end{lem}

\begin{proof}
Let $q$ be the index of $X$, i.e.\ $K_X = -qA$. We have $K_E = (a+1)A|_E$ and so, by the adjunction formula, we get
\[  (a+1)A|_E = (K_X+E)|_E = -(q-e)A|_E \]
and hence $q=a+e+1$. Since $q>e$, by the Kodaira vanishing theorem we get that $H^1(X,(m-e)A) = H^1(X, K_X+(m+q-e)A)=0$ for all $m\geq0$. It follows from the standard short exact sequence:
\[ 0 \to \sO_X((m-e)A) \to \sO_X(mA) \to \sO_E((mA)|_E) \to 0 \]
that any such $(X,A)$ has Hilbert series:
\[ P_{X,A}(t) = \frac{P_{E,\sO_E(r)}(t)}{1-t^e} = \frac{1 - t^{ra}}{(1-t)^2(1-t^a)(1-t^{ra-1})(1-t^e)} \]

(i) If $e=ra$ then $X= \PP(1,1,a,ra-1)$ and $X$ is terminal if and only if $\tfrac1{ra-1}(1,1,a)$ is at worst a terminal singularity. This happens if and only if 
\[ a+1 \equiv 0 \mod ra-1 \quad \implies \quad a+1\geq ra-1 \quad \implies \quad a(r-1) \leq 2 \]
and the only solutions for $(r,a)$ are $(2,1)$, $(3,1)$ and $(2,2)$.

(ii) If $e\neq ra$ then $X$ can be written as a hypersurface 
\[ X_{ra} \subset \PP(1,1,a,ra-1,e)_{u,x,y,z,t} \] 
and a necessary condition for $X$ to have terminal singularities at $P_e$ is that $e < ra$. Therefore bounding which numerical cases occur is equivalent to bounding $ra$. Note that if $ra>2$ then $P_z\in X$ is necessarily a $\tfrac1{ra-1}(1,a,e)$ cyclic quotient singularity and we must impose conditions to ensure it is terminal. One of the following must occur:
\begin{enumerate}
\item $a+1=ra-1\implies a(r-1)=2$ and hence $r\leq 3$, $a\leq 2$.

\item $e+1=ra-1\implies e = ra - 2$. We claim that $e \leq \tfrac12ra$ and $ra\leq4$. If not then $e>\tfrac12ra>2$ implying $P_t\in X$, which is a hyperquotient singularity with weights $\tfrac1e(1,1,a,1;2)$. This is never terminal for $e > 2$.

\item $a+e=ra-1\implies e = ra - a - 1$. By a similar calculation we get $e \leq \tfrac12ra$, which implies $(r-2)a\leq2$ and $r\leq4$. If $r=3$ or $4$ then $a\leq 2$. If $r=2$ then $X$ has a $\tfrac1{a-1}(1,1,a,a;2)$ hyperquotient singularity which is only terminal if $a\leq3$.
\end{enumerate}

We have reduced to either ($r\leq4$, $a\leq2$) or $(r,a)=(2,3)$. Checking all of these possible cases gives the list in Table \ref{possibleCases}.
\end{proof}

In particular we note that $X$ has a $\tfrac1{ra-1}$ quotient singularity, where the $\tfrac1{ra-1}(1,r)$ point of $E$ is supported, and that $X$ is smooth along $E$ elsewhere.

 \begin{table}
\arraycolsep=1.4pt\def\arraystretch{1.6}
\caption{The possible cases for $E\subset X$ and the numerical invariants of Steps 4 \& 5.}
\begin{center}
\begin{tabular}{ccccccccc}
$X$ & $(r,a)$ & $q$ & $e$ & $q'$ & $l$ & $d$ & $(ra-1)\mid d\:?$ \\ \hline
$\PP^3$ & $(2,1)$ & 4 & 2 & 2 & 3 & 7	& $\checkmark$ \\
$\PP(1^3,2)$ & $(3,1)$ & 5 & 3	& 2 & 5 & 14 	& $\checkmark$ \\
$\PP(1^2,2,3)$ & $(2,2)$ & 7 & 4 		& 3 & 5 & 13 & $\times$ \\ \hline
$X_2\subset \PP^4$ & $(2,1)$ & 3 & 1 	& 2 & 3 & 5	& $\checkmark$ \\
$X_3\subset \PP(1^4,2)$ & $(3,1)$ & 3 & 1 	& 2 & 5 & 8	& $\checkmark$ \\
$X_4\subset \PP(1^3,2,3)$ & $(2,2)$ & 4 & 1 	& 3 & 5 & 7 & $\times$ \\
$X_4\subset \PP(1^3,2,3)$ & $(4,1)$ & 4 & 2 	& 2 & 7 & 15 & $\checkmark$ \\
$X_4\subset \PP(1^2,2^2,3)$ & $(2,2)$ & 5 & 2 	& 3 & 5 & 9 	& $\checkmark$ \\
$X_6\subset \PP(1^2,2,3,5)$ & $(2,3)$ &6 &  2 	& 4 & 7 & 11 & $\times$ \\ 
$X_6\subset \PP(1^2,2,3,5)$ & $(3,2)$ &6 &  3 	& 3 & 8 & 17 & $\times$ \\ 
\end{tabular}
\end{center}
\label{possibleCases}
\end{table}%

\paragraph{Step 3.} We let $\Gamma\subset E\subset X$ be an irreducible curve of degree $d$ (where $d$ will be chosen in Step 5) passing through the point $P\in X$ such that:
\begin{description}
\item[(i)] $\Gamma$ is contained in the smooth locus of $X$.
\end{description} 
If $ra>2$ then, since $E\subset X$ passes through exactly one singular point of $X$ (the $\tfrac1{ra-1}$-quotient point), this holds if and only if $\Gamma$ avoids the $\tfrac1{ra-1}$-quotient point of $E$. A necessary condition is that $(ra-1)\mid d$ (which is also trivially true if $ra=2$).
\begin{description}
\item[(ii)] $\Gamma$ is smooth apart from an `appropriately singular' point at $P\in X$.
\end{description} 
Here `appropriately singular' means that there should exist a Mori extraction:
\[ \sigma\colon (F\subset X') \to (\Gamma\subset X) \] 
(i.e.\ a divisorial extraction in the Mori category of terminal 3-folds) with exceptional divisor $F$, such that:
\begin{itemize}
\item $X'$ has exactly one singularity of index $>1$ which is of the form $\tfrac1r(1,1,-1)$,
\item $\sigma$ induces an isomorphism $E'\cong E$, where $E'$ is the birational transform of $E$. 
\end{itemize}
The extraction $\sigma$ is given by the blowup of the symbolic power algebra of the ideal sheaf $\sI_{\Gamma/X}$ (c.f.\ \cite{duc1} Proposition 1.4). The type of curve singularities that satisfy these two conditions will be explained more carefully in Proposition \ref{singularProp}, but for now we assume that such a divisorial extraction exists. 

We take $\sigma\colon X'\to X$ to be the left-hand side of our Sarkisov link. 

\begin{rmk}
Even if $\Gamma\subset X$ is a reducible curve then a Mori extraction $\sigma\colon X'\to X$ with these properties still may exist. However $\sigma$ will not be an extremal extraction. Indeed the relative Picard rank $\rho_{X'/X}$ will equal the number of irreducible components of $\Gamma$. 
\end{rmk}

\paragraph{Step 4.} Let $q$ be the index of $X$, i.e.\ $K_X = -qA$, and let $E\in |eA|$. According to the proof of Lemma \ref{standardForm} we have $q=a+e+1$.

We define $q' = q - e = a+1$. Then, writing $A'=\sigma^*A$, we get
\[ K_{X'} = \sigma^*K_X + F = -qA' + F \quad \text{and} \quad E' = \sigma^*E - F = eA' - F. \]
In particular it follows that $K_{X'}=-(q'A' + E')$. We let $l = ra + r - 1$ and define $B' = A' + \frac{r}{l}E'$. Note that $\frac{l+1}{q'} = \frac{ra+r}{a+1} = r$ and hence we can write $K_{X'} = -q'B' + \frac1lE'$.

\paragraph{Step 5.} We now make the clever choice $d=l+re$ so that $B'$ will become numerically trivial when restricted to $E'$. 

On $E$ we have $dA|_E = r\Gamma$ since $P\in E$ has index $r$ and $\Gamma$ has degree $d$ (which is coprime to $r$). Moreover since $E\cong E'$ we must have $(F\cap E')\cong \Gamma$ and therefore it follows that $(dA' - rF)|_{E'}= 0$. Now we see that $\sO_{E'}(lB)=\sO_{E'}$ for this choice of $d$, since:
\[ (lB)|_{E'} = ( lA' + r E' )|_{E'} = ( (l + re)A' - r F )|_{E'} = ( dA' - r F )|_{E'} = 0. \]

\paragraph{Step 6.} Finally, by the Kodaira vanishing theorem we have that
\[ H^1(X',lB' - E')=H^1(X',K_{X'} + aA' - rK_{X'}) = 0 \]
which implies that the restriction map
\[ H^0(X',lB') \to H^0(E',\sO_{E'}) \]
is surjective. Then $|lB'|$ is a free linear system which is ample outside $E'$, numerically trivial along $E'$ and hence $B'$ is nef. 

The morphism $\pi\colon (X',B') \to (Y,B)$ defined by $Y=\Proj R(X',B')$ must be the contraction of the divisor $E'\cong\PP(1,a,ra-1)$ to a singularity of index $l$. Moreover since the discrepancy of $E'$ is $\tfrac1l$ we see that $\pi$ must be the Kawamata blowdown of $E'$ to the terminal cyclic quotient singularity $\tfrac1l(1,r,ra-1)$.

\paragraph{Conclusion.} All of the steps in this construction are valid provided we can show the existence of a divisorial extraction $\sigma\colon X'\to X$ with the properties that were claimed in condition (ii) of Step 3. In this case we will have constructed a Sarkisov link from $(X,A)$ to a $\QQ$-Fano 3-fold $(Y,B)$ of index $q'$ and degree $B^3=A'^2(A'+\tfrac{r}{l}E')=\tfrac{d}{l}A^3$.

\section{Divisorial extractions from singular curves}

We consider the six remaining cases of Table \ref{possibleCases} which satisfy $(ra-1)\mid d$ and check that the construction actually does work in these cases. We will call these six cases A.1-4 and B.1-2 according to the order in which they appear in Table \ref{Xtable}.

 \begin{table}
\arraycolsep=1.2pt\def\arraystretch{1.6}
\caption{The $\QQ$-Fano 3-folds $(X,A)$ which give genuine Sarkisov links.}
\begin{center}
\begin{tabular}{c|ccccm{2.2cm}ccm{2.4cm}c}
Case & $X$ & $\codim$ & $q_X$ & $A^3$ & \centering $\sB_X$ & $(r,a)$ & $d$ & \centering Singularity type & \\ \cline{1-9}
A.1 &$\PP^3$ & 0 & 4 & 1 & \centering $\emptyset$ &  $(2,1)$ & 7 & \centering  $\Gamma_{(3)}$ & 	 \\
A.2 & $X_2\subset \PP^4$  & 1 & 3 & 2 & \centering $\emptyset$ & $(2,1)$  & 5 & \centering  $\Gamma_{(3)}$ &  \\ 
A.3 & $\PP(1^3,2)$ & 0 & 5 & $\tfrac12$ & \centering $\tfrac12(1,1,1)$ & $(3,1)$ & 14 & \centering $\Gamma_{(1,3)}$ or $\Gamma_{(4,0)}$ &  \\
A.4 & $X_4\subset \PP(1^2,2^2,3)$ & 1 & 5 & $\tfrac13$ & \centering $2\times\tfrac12(1,1,1)$, $\tfrac13(1,2,2)$  & $(2,2)$  & 9 & \centering $\Gamma_{(3)}$ & \\ \cline{1-9}
B.1 & $X_3\subset \PP(1^4,2)$  & 1 & 3 & $\tfrac32$ & \centering $\tfrac12(1,1,1)$ &$(3,1)$ & 8 & \centering $\Gamma_{(4,0)}$ &   \\ 
B.2 &  $X_4\subset \PP(1^3,2,3)$   & 1 & 4 & $\tfrac23$ & \centering  $\tfrac13(1,1,2)$ & $(4,1)$  & 15 & \centering $\Gamma_{(5,0,0)}$ &
\end{tabular} 
\end{center}
\label{Xtable}
\end{table}%

\subsection{Curves in type $A$ Du Val singularities}

We now change focus slightly and assume that we are in the \emph{local} setting. We take $P\in S$ to be the germ of a type $A_{r-1}$ Du Val surface singularity and we fix isomorphisms 
\[ (P\in S) \: \cong \: (0 \in \CC^2_{\alpha,\beta}) / \tfrac1r(1,-1)  \: \cong \: \big(0\in V(xz - y^r)\subset \CC^3_{x,y,z}\big) \]
where $(x,y,z)=(\alpha^r,\alpha\beta,\beta^r)$ are the invariant generators for the $\tfrac1r(1,-1)$ cyclic group action. As is well known, $P\in S$ has minimal resolution 
\[  \mu \colon ( D \subset \widetilde{S}) \to (P\in S) \]
where the exceptional divisor $D=\bigcup_{i=1}^{r-1} D_i$ is a chain of $-2$-curves of length $r-1$. In particular $D_i\cong\PP^1$ for all $i$ and coordinates on $D_i$ can be given in terms of the ratio $x/y^i=y^{r-i}/z$ or equivalently in terms of the orbinates $\alpha^{r-i}/\beta^i$. Given a curve $P\in \Gamma\subset S$ we write $\widetilde{\Gamma}$ for the birational transform of $\Gamma$ on $\widetilde{S}$.

\begin{defn}
We call  a curve germ $P\in\Gamma\subset S$ a \emph{singularity of type $\Gamma_{(a_1,\ldots,a_{r-1})}$} if $\widetilde{\Gamma}$ has $a_i$ branches intersecting $D_i$ transversely, for all $i$. (In particular $\widetilde{\Gamma}$ avoids the intersection points $D_{i-1}\cap D_i$.)
\end{defn}

\paragraph{The orbifold equation.} The \emph{orbifold equation} of $\Gamma$ is the equation $\gamma\in\CC[\alpha,\beta]$ defining $q^{-1}\Gamma \subset \CC^2$, the preimage of $\Gamma$ under the quotient map $q\colon \CC^2\to S$. If $P\in \Gamma$ is a singularity of type $\Gamma_{(a_1,\ldots,a_{r-1})}$ then the orbifold equation of $\Gamma$ factors analytically:
\[ \gamma(\alpha,\beta) = \prod_{i=1}^{r-1} \prod_{j=1}^{a_i} ( \lambda_{ij} \alpha^{r-i} - \mu_{ij} \beta^i ) \]
for some functions $\lambda_{ij},\mu_{ij}\in\CC[[x,y,z]]$ according to the branches of $\widetilde\Gamma$. We will usually multiply out this expression and collect together terms to write $\gamma$ as:
\[ \gamma(\alpha,\beta) = \sum_{j=0}^{a_1+\cdots+a_{r-1}} c_j \alpha^{m_j} \beta^{n_j}  \]
where $c_j\in\CC[x,y,z]$ are invariant polynomials with constant term $c_{j,0}\neq0$ and the points $(m_j,n_j)$ lie on the boundary of the Newton polygon $\Newt(\gamma)$. The faces of $\Newt(\gamma)$ have slope $-\tfrac{r-i}{i}$ for $i=1,\ldots,r-1$. If we let $A_i := \sum_{j\leq i}a_j$, then restricting $\gamma$ to a face gives 
\[ \sum_{j=A_{i-1}}^{A_i} c_{j,0} \alpha^{m_j} \beta^{n_j} = \alpha^{m_{A_{i-1}}}\beta^{n_{A_i}}\gamma_i(\alpha,\beta) \] 
where $\gamma_i$ is a homogeneous polynomial in $\alpha^{r-i},\beta^i$ of degree $a_i$ whose roots give the intersection points of $\widetilde\Gamma \cap D_i$. By assumption these roots are distinct and both $\alpha\nmid \gamma_i$ and $\beta\nmid \gamma_i$. We can consider \emph{degenerations} of $\Gamma$ by allowing the $\gamma_i$ to pick up multiple roots and by allowing some of the coefficients $c_{j,0}$ to vanish. 

\begin{eg}
For the $A_2$ singularity, the singularity types $\Gamma_{(1,3)}$ and $\Gamma_{(4,0)}$ have resolutions that look like the following:
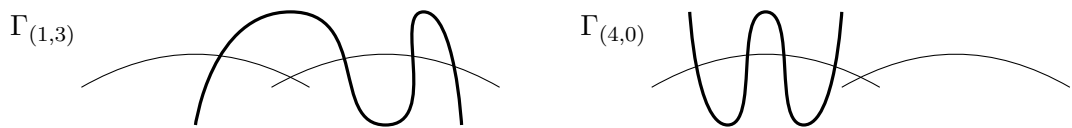
\begin{center}\begin{tikzpicture}
	\node at (-0.5,0.75) {$\Gamma_{(1,3)}$};
	\draw (0,0) to[out=30,in=150] (3,0);
	\draw (2.5,0) to[out=30,in=150] (5.5,0);
	\draw[very thick] plot [smooth, tension=1.5] coordinates {(1.5,-0.5) (2.75,1) (4,-0.5) (4.5,1) (5,-0.5)};
	
	\node at (7,0.75) {$\Gamma_{(4,0)}$};
	\draw (7.5,0) to[out=30,in=150] (10.5,0);
	\draw (10,0) to[out=30,in=150] (13,0);
	\draw[very thick] plot [thick, smooth, tension=1.5] coordinates {(8,1) (8.5,-0.5) (9,1) (9.5,-0.5) (10,1)};
\end{tikzpicture}\end{center}
In both cases a format for the orbifold equation is given in Proposition \ref{singularProp}.
\end{eg}

\paragraph{Weighted $A_{r-1}$ singularities.} Going back to Table \ref{Xtable}, the orbinates $\alpha,\beta$ at $P\in E$ are naturally weighted with weights $\tfrac1r, \tfrac{ra-1}{r}$ (and hence $x,y,z$ have weights $1,a,ra-1$). If $\Gamma\subset E$ is a curve of degree $d$ then each term appearing in the orbinate equation $\gamma$ must have degree $\leq\tfrac{d}{r}$ with respect to these weights. Since we chose $d=l+re=qr-1\equiv -1 \mod r$, it follows that $\Gamma$ has an orbifold equation of the form $\gamma = \alpha^{r-1}\phi + \beta\psi$, where $\phi,\psi\in\CC[x,y,z]$ have degree $\leq q-1$.

\subsection{`Appropriately singular' curves}

Proposition \ref{singularProp} describes precisely what is meant by the statement `appropriately singular' in Step 3(ii) of the construction.

\begin{prop}
Suppose that $P\in\Gamma \subset S\subset U$ where $P\in S$ is an $A_{r-1}$ Du Val singularity with $r\leq4$, $P\in U$ is a smooth 3-fold and $\Gamma$ is a curve of degree $d\equiv r-1\mod r$. Suppose that there exists a Mori extraction $\sigma\colon (F\subset U') \to (\Gamma\subset U)$ where $U'$ has a single high index singularity of type $\tfrac1r(1,1,-1)$. Then the birational transform $S'=\sigma^{-1}S$ is isomorphic to $S$ and $P\in\Gamma$ is one of the following singularity types (up to a degeneration).
\begin{description}
\item[Type $A_1$.] The only possibility is $P\in \Gamma_{(3)}$ with multiplicity 3 and orbinate equation:
\[ \gamma_{(3)}  = a\alpha^3 + b\alpha^2\beta + c\alpha\beta^2 + d\beta^3  \]
\item[Type $A_2$.] There are two possibilites:
\begin{enumerate}
\item[(i)] $P\in \Gamma_{(1,3)}$ with multiplicity 4 and orbinate equation:
\[ \gamma_{(1,3)} = a\alpha^5 + b\alpha^3\beta + c\alpha^2\beta^3 + d\alpha\beta^5 + e\beta^7 \]
\item[(ii)] $P\in \Gamma_{(4,0)}$ with multiplicity 4 and orbinate equation:
\[ \gamma_{(4,0)} = a\alpha^8 + b\alpha^6\beta + c\alpha^4\beta^2 + d\alpha^2\beta^3 + e\beta^4  \]
\end{enumerate}
\item[Type $A_3$.] Suppose moreover that $\alpha,\beta$ have weights $\tfrac14,\tfrac34$ and that $\Gamma$ has degree $\leq 15$. Then the only possibility is $P\in \Gamma_{(5,0,0)}$ with multiplicity 5 and orbinate equation:
\[ \gamma_{(5,0,0)} = a\alpha^{15} + b\alpha^{12}\beta + c\alpha^9\beta^2 + d\alpha^6\beta^3 + e\alpha^3\beta^4 + f\beta^5  \] 
\end{description}
\label{singularProp}
\end{prop}

\begin{rmk}
If we remove the condition on the degree in the $A_3$ case then, with more work, it is possible to show that the curve singularities satisfying the conditions of the Theorem are precisely $\Gamma_{(5,0,0)}$, $\Gamma_{(1,0,4)}$, $\Gamma_{(0,3,1)}$ and their degenerations. However the last two singularity types cannot occur in case B.2 since $\gamma_{(1,0,4)}$ and $\gamma_{(0,3,1)}$ both contain terms of degree $>\tfrac{15}4$. 
\end{rmk}

\begin{proof}
These statements can be checked by constructing divisorial extractions by unprojection, as in the style of \cite{duc1}. Prokhorov \& Reid treat the $A_1$ case \cite{prokreid} \S6.1, and both the $A_1$ and $A_2$ cases appear in \cite{duc1} \S3. Therefore we only consider the $A_3$ case.

Since we are assuming that $\deg\gamma\leq\tfrac{15}{4}$, we can write the orbifold equation of $\Gamma$ as
\[ \gamma(\alpha,\beta) = \alpha^3 \phi(x,y) + \beta \psi(y,z) \]
where $\phi,\psi$ are functions of degree $\leq3$ and the variables $x,y,z$ have weights $1,1,3$. Therefore we can write $\phi$ and $\psi$ as
\[ \phi = ax + by + cx^2 + dxy + ey^2 + \phi_3, \quad \psi = fy + gy^2 + hz + \psi_3 \]
where $a,\ldots,h\in\CC$ are constants and $\phi_3,\psi_3$ only contain terms of degree 3. We can assume that $h\neq0$, else $\alpha\mid \gamma$ and hence $\gamma$ is reducible. To show that $\Gamma$ has the form stated in the Theorem we must show that $a,\ldots,g=0$.

As a subvariety of $U$ our curve $\Gamma$ is defined by the minors:
\[ \bigwedge^2\begin{pmatrix}
x & y^3 & -\psi(y,z) \\
y & z & \phi(x,y)
\end{pmatrix} = 0 \]
For $\lambda,\mu\in\CC$, let $H_{\lambda,\mu}$ be the hyperplane section containing $\Gamma$ which is given by the equation:
\[ h_{\lambda,\mu} := xz - y^4 + \lambda(x\phi + y\psi) + \mu(y^3\phi + z\psi) = 0 \]
If, for some $\lambda,\mu$, $P\in H_{\lambda,\mu}$ is Du Val singularity of type $A_{\leq2}$ then, by changing coordinates and following the constructions in the $A_{\leq2}$ cases, the divisorial extraction from $\Gamma$ will have a singularity of index $\leq3$. Therefore in order for $U'$ to have a $\tfrac14(1,1,3)$ singularity we require that $P\in H_{\lambda,\mu}$ is a type $A_3$ singularity or worse $\forall\lambda,\mu\in\CC$. By the finite determinancy of Du Val singularities, this happens if and only if $h_{\lambda,\mu}$ has a factorisation as a product $h_{\lambda,\mu}\equiv XZ \mod \frakm^4$ for some $X,Z$, $\forall\lambda,\mu\in\CC$.

First we check which conditions are needed for $h_{\lambda,\mu}$ to factor as a product mod $\frakm^3$:
\[  h_{\lambda,\mu} \equiv xz + \lambda x (ax + by) + \lambda y(fy + hz) + \mu z(fy + hz) \mod\frakm^3 \]                        
This factors if and only if the discriminant of this quadratic form vanishes identically, i.e.: 
\[ \lambda \big( a(h\lambda - f\mu)^2 + b^2h\lambda\mu - b(h\lambda + f\mu) + f \big) = 0, \quad \forall\lambda,\mu\in\CC \]
Since we are assuming $h\neq0$ this implies $a=b=f=0$.

Now it is possible to check that $h_{\lambda,\mu}$ can be factorised mod $\frakm^4$ as follows:
\begin{align*}
 h_{\lambda,\mu} &\equiv xz + \lambda x (cx^2 + dxy + ey^2) + \lambda y(gy^2 + hz) + \mu z(gy^2 + hz) \\
 &\equiv XZ - \lambda(ch^3 \lambda^3 - dh^2 \lambda^2 + eh \lambda - g)y^3 & \mod \frakm^4
\end{align*} 
where
\begin{align*}
X & = x + h(\lambda y + \mu z) + g\mu y^2 - h \lambda\mu (cx^2 + dxy + ey^2) + h^2 \lambda^2\mu y (cx + dy) - ch^3 \lambda^3\mu y^2 \\
Z &= z + \lambda (cx^2 + dxy + ey^2) - h \lambda^2 y(cx + dy)  + ch^2 \lambda^3 y^2 
\end{align*}
Therefore $h_{\lambda,\mu}$ factorises mod $\frakm^4$ if and only if 
\[ ch^3 \lambda^3 - dh^2 \lambda^2 + eh \lambda - g = 0, \quad \forall\lambda\in\CC \]
and this implies that $c=d=e=g=0$.

We have proved that it is \emph{necessary} for $\Gamma$ to be of the form stated in the Theorem. To prove that the symbolic blowup of such a curve $\Gamma$ actually does have a $\tfrac14(1,1,3)$ singularity and that $S'\cong S$ we can construct $U'$ explicitly using the serial unprojection method of \cite{duc1}. This gives $\sigma\colon U' \to U$ as the $\Proj$ of a Gorenstein ring which has relative codimension 5 and 14 equations:
\[ \sigma\colon U' \subset U\times \PP(1,1,1,2,3,4) \to U \]
where $U'$ has a $\tfrac14(1,1,3)$ singularity at the last coordinate point. In the general case the central fibre $\sigma^{-1}(P)$ consists of five rational curves meeting at the $\tfrac14$-point. If $\gamma$ degenerates to obtain a root of multiplicity $m$, then $m$ of the curves in the central fibre come together and $U'$ picks up an isolated $cA_{m-1}$ point along this curve.
\end{proof}

\subsection{Proof of Theorem \ref{mainThm}}
\label{proofOfMainThm}

The last thing left to check in the proof of Theorem \ref{mainThm} is the existence in each case of a divisorial extraction $\sigma \colon X' \to X$ which has the properties claimed in Step 3(ii) of the construction. This is equivalent to checking that a degree $d$ curve $\Gamma\subset E\subset X$ can have one of the singularity types appearing in Proposition \ref{singularProp}. 

\paragraph{Possible singularity types.}
We simply have to check whether it is possible to take coefficients of degree $\geq0$ in the relevant orbifold equation of Proposition \ref{singularProp}. Only case B.1 and the singularity type $\Gamma_{(1,3)}$ is impossible. In the other cases we get the following singularity types by taking coefficients of the following degrees:
\begin{description}
\item[(A.1)] Singularity type $\Gamma_{(3)}$ if we take $a_2, b_2, c_2, d_2 \in \CC[ u_1, x_1, y_1, z_1]$.
\item[(A.2)] Singularity type $\Gamma_{(3)}$ if we take $a_1, b_1, c_1, d_1 \in \CC[ u_1, x_1, y_1, z_1]$.
\item[(A.3.i)] Singularity type $\Gamma_{(1,3)}$ if we take $a_3, b_3, c_2, d_1, e_0 \in \CC[ u_1, x_1, y_1, z_2]$.
\item[(A.3.ii)] Singularity type $\Gamma_{(4,0)}$ if we take $a_2, b_2, c_2, d_2, e_2 \in \CC[ u_1, x_1, y_1, z_2]$.
\item[(A.4)] Singularity type $\Gamma_{(3)}$ if we take $a_3, b_2, c_1, d_0 \in \CC[ u_1, x_1, y_2, z_3]$.
\item[(B.1)] Singularity type $\Gamma_{(4,0)}$ if we take $a_0, b_0, c_0, d_0, e_0 \in \CC[ u_1, x_1, y_1, z_2]$.
\item[(B.2)] Singularity type $\Gamma_{(5,0,0)}$ if we take $a_0, b_0, c_0, d_0, e_0, f_0 \in \CC[ u_1, x_1, y_1, z_3]$.
\end{description}

\paragraph{Irreducibility.} It is not hard to check that the generic curve with such a choice of coefficents is irreducible in each of the cases A.1-4. However in the B.1 case, since the coefficients are all constant, the orbifold equation for $\Gamma$ is a quartic polynomial in $\alpha^2$ and $\beta$. Thus $\Gamma$ decomposes as a union of four irreducible smooth curves (which meet non-transversely) corresponding to the roots of this quartic. Since $\rho_{X'}=5$ we must have $\rho_Y=4$. Similarly in the B.2 case $\Gamma$ has five irreducible smooth components and $\rho_Y=5$.

\section{Constructing $Y$ by unprojection}

In their treatment of the A.1 and A.2 cases, Prokhorov \& Reid \cite{prokreid} \S6.4 also explain how to construct these two families explicitly by unprojection. Similar unprojection constructions are also possible for the remaining cases of Table \ref{Ytable} which are essentially equivalent to the unprojection construction of the divisorial extraction in the relative setting. We will only construct and study the two families in the A.3 case. The enthusiast will enjoy working out the remaining cases A.4, B.1, B.2 for his or herself. 

\subsection{The two A.3 families}  \label{constructingY}

By following the method in \S1 for the A.3 case we construct a $\QQ$-Fano $(Y,B)$ of index $q'=2$, degree $B^3=\tfrac75$ and $\dim|B|=3$ which has a single $\tfrac15(1,2,4)$ quotient singularity. By the Ice Cream formula \cite{iceCream} the Hilbert series of $Y$ is:
\[ P_{Y,B}(t) = \frac{1+t^2}{(1-t)^4} + \frac{t^2+t^4}{(1-t)^3(1-t^5)}  =  \frac{N(t)}{(1-t)^3(1-t^2)^2(1-t^3)(1-t^4)(1-t^5)} \]
where the numerator is the degree 17 polynomial:
\begin{align*} 
N(t) &= 1 - 2t^4 - 2t^5 - 2t^6 + 2t^7 + 3t^8 + 3t^9 + 2t^{10} - 2t^{11} - 2t^{12} - 2t^{13} + t^{17} \\
 &=  1 - (2t^4 + 2t^5 + 3t^6 + t^7 + t^8) + (t^6 + 3t^7 + 4t^8 + 4t^9 + 3t^{10} + t^{11}) - \cdots 
\end{align*}
where $\cdots$ denotes the usual palindromic Gorenstein symmetry. We claim that $Y$ has three masked relations in degrees $6,7,8$ in addition to the six relations in degrees $4^2,5^2,6^2$ which can be read off from $N(t)$. This will be justified by the unprojection calculation below and it suggests that $Y$ has a presentation in the familiar Gorenstein codimension 4 format with 9 equations and 16 syzygies:
\[ Y_{4^2,5^2,6^3,7,8}\subset \PP(1^3,2^2,3,4,5)_{u_1,x_1,y_1,z_2,\xi_2,\nu_3,\zeta_4,\theta_5} \]

We can assume that the orbinates at the $\tfrac15(1,2,4)$ point $P_\theta\in Y$ are given by $u_1,\xi_2,\zeta_4$. Then $Y$ has four equations of the form $x\theta,y\theta,z\theta,\nu\theta=\cdots$ which have degrees $6,6,7,8$ and eliminate these four variables at this point. We consider the projection from $P_\theta\in Y$:
\[ \widehat\theta \colon (P_\theta \in Y) \dashrightarrow (\Pi \subset Y')\subset \PP(1^3,2^2,3,4) \]
This corresponds to the $(1,2,4)$-weighted blowup of $P_\theta$ followed by a small contraction of some curves in the tangent space $T_\theta Y\cap Y$. This takes us out of the Mori category since $Y'$ contains a line of $\tfrac12(1,1)$ singularities with a dissident $\tfrac14(1,2,3)$ point at $P_\zeta\in Y'$ (in addition to the non-$\QQ$-factorial singularities given by the small contraction). In particular the orbinates at $P_\zeta\in Y'$ are $u_1,\xi_2,\nu_3$ and hence $Y'$ has equations $x\zeta,y\zeta,z\zeta=\cdots$ in degrees $5,5,6$ which eliminate $x,y,z$ at $P_\zeta\in Y$. We have found the masked equations in degrees $6,7,8$ as claimed.

$Y'$ is a special 3-fold in codimension 3 containing the plane $\Pi = \PP(1,2,4)_{(u:\xi:\zeta)}$ which is the exceptional divisor of the projection. We reverse this process to construct $Y$ by \emph{unprojecting} $\Pi\subset Y'$.

\paragraph{The special 3-fold $Y'$.}
The projection gives a codimension 3 variety $Y'$ defined by the maximal Pfaffians of a $5\times5$ skew matrix $M$ with entries of the following weights:
\[ Y' = \left( \Pf \left( \begin{smallmatrix} 4&3&3&2 \\ &3&3&2 \\ &&2&1 \\ &&&1 \end{smallmatrix} \right) = 0 \right) \subset \PP(1^3,2^2,3,4)_{u,x,y,z,\xi,\nu,\zeta} \]
Without loss of generality we can use row and column operations and rename variables to write $M$ in the form:
\[ M = \begin{pmatrix}
\zeta & \nu & t_1\xi + d'_3 & -e'_2 \\ 
 & -a'_3 & \nu + c'_3 & \xi \\
 & & z & y \\
 & & & x
\end{pmatrix} \]
where $a'\in\CC[u,x,y]$, $e'\in\CC[u,y,z]$ and  $b',c',t\in\CC[u,y]$ are polynomials of the indicated degree. Now we must impose some further conditions to ensure that $Y'$ contains $\Pi$, or equivalently that the equations $\Pf M =0$ are contained in the ideal $I_\Pi=(x,y,z,\nu)$. The two ways of doing this are given by the Tom and Jerry formats of \cite{bkr}. In our case $M$ is in Tom$_2$ format for the ideal $I_\Pi$ if $d',e',t\in I_\Pi$ (i.e.\ all entries except those in the second row and column are in $I_\Pi$) and in Jer$_{34}$ format if $a',c',d',t\in I_\Pi$ (i.e.\ all entries in the third and fourth row and column are in $I_\Pi$). In either case we have $t=t_1(u,y)\in I_\Pi$ and hence we can take $t=y$.

\paragraph{Recovering $\Gamma\subset X$.} 
In direct analogy to Prokhorov \& Reid's example (see \cite{prokreid} \S6.4 and equation (6.4.4)), by projecting from $P_\zeta\in Y'$ we see that we can construct $Y$ as a double unprojection, starting from the complete intersection:
\[ \left( \begin{pmatrix} 
x & y^2 & -(d'y + e'z) \\
y & z & a'x + c'y
\end{pmatrix}
\begin{pmatrix} 
\nu \\ -\xi \\ 1
\end{pmatrix} =0 \right) \quad \subset \quad \PP(1^3,2^2,3)_{u,x,y,z,\xi,\nu} \]
with first unprojection ideal $(x,y,z)$. Taking the minors of the $2\times 3$ matrix appearing in this format recovers the equations of the curve $\Gamma\subset X=\PP(1,1,2,3)_{u,x,y,z}$ which was blown up in the construction of the Sarkisov link.

\subsubsection{The Tom$_2$ family $\mathcal{T}$} 
For $M$ to be in Tom$_2$ format we must take $d',e'\in I_\Pi$ in addition to $t=y$. We can take $d_3'=d_2y$ and $e_2'=e_1y+f_0z$. Then $M$ and the unprojection equations for the ideal $I_\Pi$ are:
\[ \begin{pmatrix}
\zeta & \nu & y(\xi + d) & -(ey + fz) \\
 & -a' & \nu + c' & \xi \\
 & & z & y \\
 & & & x
\end{pmatrix} \quad \begin{array}{rl} 
\vspace{0.1cm} x\theta & \!\!\!  = \xi^2(\xi + d) + e\xi(\nu+c') + f(\nu+c')^2 \\ 
\vspace{0.1cm} y\theta & \!\!\!  = \xi\zeta - a'f(\nu + c') \\ 
\vspace{0.1cm} z\theta & \!\!\! = (\nu + c')(\zeta + a'e) + a'\xi(\xi + d) \\ 
\nu\theta & \!\!\! = \zeta(\zeta + a'e) + a'^2f(\xi + d)
\end{array} \]
We note that in this case the curve $P\in\Gamma\subset X$ has the orbifold equation:
\[ \gamma_\mathcal{T}(\alpha,\beta) = a'_3\alpha^5 + c'_3\alpha^3\beta + d_2\alpha^2\beta^3 + e_1\alpha\beta^5 + f_0\beta^7 \]
which is the orbifold equation for a singularity of type $\Gamma_{(1,3)}$.


\subsubsection{The Jer$_{34}$ family $\mathcal{J}$}
For $M$ to be in Jer$_{34}$ format we must take $a',c',d'\in I_\Pi$ in addition to $t=y$. This time we can take $a_3'=a_2x+b_2y$, $c_3'=c_2y$ and $d_3'=d_2y$. Then $M$ and the unprojection equations for the ideal $I_\Pi$ are:
\[ \begin{pmatrix}
\zeta & \nu & y(\xi + d) & -e' \\
 & -(ax+by) & \nu + cy & \xi \\
 & & z & y \\
 & & & x
\end{pmatrix} \quad \begin{array}{rl} 
\vspace{0.1cm} x\theta & \!\!\! = \xi^3 + d\xi^2 + ce'\xi + e'(\zeta + be') \\
\vspace{0.1cm} y\theta & \!\!\! = \xi\zeta - ae'^2 \\
\vspace{0.1cm} z\theta & \!\!\! = \nu(\zeta + be') + (ax+by)\xi(\xi + d) + c\xi\nu + ae'y(\xi+d) \\
\nu\theta & \!\!\! = \zeta(\zeta + be') + ae'(\xi^2 + d\xi + ce')
\end{array} \] 
In this case the orbifold equation of the curve $P\in\Gamma\subset X$ is
\[ \gamma_\mathcal{J}(\alpha,\beta) = a_2\alpha^8 + b_2\alpha^6\beta + c_2\alpha^4\beta^2 + d_2\alpha^2\beta^3 + e_2'\beta^4 \]
which is the orbifold equation for a singularity of type $\Gamma_{(4,0)}$.


\subsection{A common degeneration}

By construction we see that every $\QQ$-Fano 3-fold $Y$ with the numerical invariants given in \S\ref{constructingY} belongs to one of the two families $\mathcal{T}$ and $\mathcal{J}$. Moreover these correspond precisely to two families of $\QQ$-Fano 3-folds which admit a Sarkisov link to $X=\PP(1,1,1,2)$ ending in the divisorial contraction to $\Gamma\subset X$, a degree 14 curve either with orbifold equation $\gamma_\mathcal{T}$ and a singularity of type $\Gamma_{(1,3)}$, or with orbifold equation $\gamma_\mathcal{J}$ and a singularity of type $\Gamma_{(4,0)}$. 

The format for $Y$ is determined by the curve $\Gamma\subset X$, which in turn is determined by the orbifold equation, up to rescaling by a constant. Counting the number of monomials in $\gamma_\mathcal{T}$ and $\gamma_\mathcal{J}$ we see that the dimension of each family is given by:
\[ \dim \mathcal{T} = 37 - 1 = 36 \quad \text{and} \quad \dim \mathcal{J}= 35-1=34. \]
Moreover these two families have a common degeneration corresponding to the curve singularity $P\in \Gamma$ of type $\Gamma_{3,2}$ with orbifold equation:
\[ \gamma_{3,2}(\alpha,\beta) = a_2\alpha^8 + b_2\alpha^6\beta + c_2\alpha^4\beta^2 + d_2\alpha^2\beta^3 + e_1\alpha\beta^5 + f_0\beta^7  \]
which is obtained when $M$ is simultaneously in both Tom$_2$ and Jer$_{34}$ format (i.e.\ take either $a_3' = a_2x+b_2y$, $c'_3=c_2y$ in the format for $\mathcal{T}$ or $e_2'=e_1y+f_0z$ in the format for $\mathcal{J}$). Similarly this intersection has $\dim (\mathcal{T}\cap \mathcal{J})=31$ and these families fit together as in Figure \ref{twoComps}. 
\begin{figure}[h]
\begin{center}
\begin{tikzpicture}[scale=1]
	\node at (1.6,1.8) {\small $\mathcal{T}$ : $\Gamma_{(1,3)}$};
	\node at (1.6,1.3) {\small $\dim = 36$};
	
	\node at (3,3.4) {\small $\mathcal{T}\cap\mathcal{J}$ : $\Gamma_{(3,2)}$};
	\node at (3,2.9) {\small $\dim = 31$};
	\draw[->] (3,2.7) -- (3,2.4);
	
	\node at (4.4,1.8) {\small $\mathcal{J}$ : $\Gamma_{(4,0)}$};
	\node at (4.4,1.3) {\small $\dim = 34$};

	\draw[very thick] (3,2.2) -- (3,-0.2);
	\draw[thick] (0,1) -- (3,0) -- (3,2) -- (1,3) -- cycle;
	\draw[thick] (6,1) -- (3,0) -- (3,2) -- (5,3) -- cycle;
\end{tikzpicture}
\caption{The two families $\mathcal{T}$ and $\mathcal{J}$.}
\label{twoComps}
\end{center}
\end{figure}
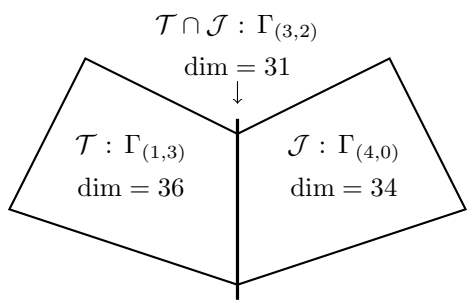

Now suppose that $Y$ is a general member of the intersection $\mathcal{T}\cap \mathcal{J}$. Since $M$ is in both Tom$_2$ and Jer$_{34}$ format simultaneously it follows from looking at either format, $\mathcal{T}$ or $\mathcal{J}$, that all of the nine equations defining $Y$ are contained in $\frakm_{u}^2$, where $\frakm_u$ is the maximal ideal of the point $P_u\in Y$. Hence $P_u\in Y$ is an index 1 singularity which has embedding dimension 7 and therefore cannot be terminal. It would be interesting to know if such a singularity is canonical.

\end{document}